\documentclass[
  11pt     %
]{article}

\usepackage{color}
\definecolor{black}{rgb}{0,0,0}

\definecolor{red}{rgb}{1,0,0}

\definecolor{blue}{rgb}{0,0,1}

\usepackage{amstext,amsthm,amssymb,amsmath}
\usepackage{bm}
\usepackage{graphicx}
\usepackage{subfigure}
\usepackage{wrapfig}
\usepackage{fullpage}
\usepackage{color}
\usepackage{multirow}
\usepackage{tabulary}
\usepackage{booktabs}
\usepackage{enumerate}
\usepackage{epstopdf}

\newtheorem{theorem}{Theorem}[section]
\newtheorem{lemma}[theorem]{Lemma}

\newtheorem{remark}[theorem]{Remark}
\newtheorem{Definition}[theorem]{Definition}
\newcommand{\DIV}{\mathrm{div}\,}
\newcommand{\DIVinThm}{\mathrm{div}\,}

\newcommand{\prnt}[1]{\left( #1 \right)}

\newcommand{\VECT}[1]{\mathbf{#1}}

\newcommand{\norm}[1]{\left\|#1\right\|}

\newcommand{\Tnorm}[2]{{|\!|\!| #1 |\!|\!|}_{{\text{V},#2}}}
\newcommand{\TnormLer}[2]{{\bigg{|\!\bigg|\!\bigg|} #1 \bigg{|\!\bigg|\!\bigg|\;}}_{{\text{V},#2}}}

\newcommand{\Snorm}[2]{{|\!|\!| #1 |\!|\!|}_{{\text{P},#2}}}

\newcommand{\normH}[2]{\norm{#1}_{H^1\prnt{#2}}}

\newcommand{\normL}[2]{\norm{#1}_{L^2\prnt{#2}}}

\newcommand{\normI}[2]{\norm{#1}_{L^{\infty}\prnt{#2}}}

\newcommand{\p}{\VECT{p}}

\newcommand{\Om}{\Omega}

\newcommand{\avrg}[2]{{\left \langle #1 \right \rangle}_{#2}}

\title{A Generalized Multiscale
Finite Element Method for the Brinkman Equation}

\author{ Guanglian Li\thanks{Department of Mathematics  and Institute for Scientific Computation (ISC),
Texas A\&M University,
College Station, Texas 77843-3368, USA}
\and Juan Galvis\thanks{Departamento de Matem\'{a}ticas, Universidad Nacional de Colombia,
Bogot\'a D.C., Colombia}  \and  Ke Shi\footnotemark[1]}

\begin{document}
\maketitle
%
\begin{abstract}
In this paper we consider the numerical upscaling of
the Brinkman equation in the presence of high-contrast permeability fields.
We develop and analyze a  robust and efficient Generalized Multiscale Finite
Element Method (GMsFEM) for the Brinkman model.
In the fine grid, we use mixed finite element method with the velocity and pressure being
continuous piecewise quadratic and  piecewise constant finite element spaces,
respectively.
Using the GMsFEM framework we construct suitable coarse-scale spaces for the velocity and pressure
that yield a robust mixed GMsFEM.
We develop a novel approach to construct a coarse approximation  for the velocity snapshot space and
a robust small offline space for the velocity space. The stability of the mixed GMsFEM and a priori error
estimates are derived. A variety of two-dimensional numerical examples are presented to illustrate the effectiveness of the algorithm.
\end{abstract}

\section{Introduction}

In this paper, we design and analyze an efficient numerical method based on the generalized multiscale finite
element method (GMsFEM) framework for Brinkman type system of partial differential equations in the
context of mixed finite element method. The Brinkman equation is widely accepted in
the mathematical modeling of flows in heterogeneous fields, e.g., vuggy carbonate
reservoirs, low porosity filtration devices and biomedical hydrodynamic studies \cite{GHL08, KV03}.
In these applications, the simple Darcy model is inadequate
to capture their essential physics \cite{ref:Iliev2011,ref:Popov&.2007.Anaheim} and the Darcy-Stokes interface model is not feasible since the
precise information about the location and geometry of the interface between vugs and the
porous matrix as well as experimentally determined values related to the interface conditions are inaccessible.
The Brinkman flow behaves like a Darcy flow and a Stokes flow for regions with very
large permeability values and with small permeability values, respectively. Hence, in comparison
with the popular Stokes-Darcy interface model, the Brinkman model can describe both a Stokes
and a Darcy flow but without using a complex interface condition. Hence, the accuracy and efficiency
of the Brinkman flow simulation is of significant practical interest \cite{EGLW12,DB87, VOA07, XXX08}. In our earlier work \cite{dels11},
we derived homogenization results for high-contrast Brinkman flow in a periodic permeability field. We showed that
the homogenization method can simplify the high-contrast periodic Brinkman model, and the resulting solution is a good approximation of
the original Brinkman model.

In this work, we investigate the high-contrast Brinkman flow in general permeability fields instead of the periodic fields as analysed in \cite{dels11}.
Often, model reduction techniques are required for efficiently
resolving such multiscale problems. These techniques all rely on a coarse grid approximation,
obtained by discretizing the problem on a coarse grid, much coarser than the fine grid,
and a suitable coarse-grid formulation of the problem. In the literature, several different approaches have
been proposed to obtain the coarse-grid formulation, which can be roughly divided into upscaling models
\cite{dur91, weh02} and multiscale methods (see, e.g., \cite{Arbogast_two_scale_04, Chu_Hou_MathComp_10,ee03,
egw10,eh09,ehg04, Wave} and the references therein). Among existing multiscale methods, the GMsFEM
framework \cite{egw10,egh12} of recent origin has demonstrated great promise; see \cite{egw10,GhommemJCP2013,
ReducedCon,MsDG,Wave,WaveGMsFEM, egh12} for methodological developments and extensive applications.
In the GMsFEM, the coarse grid problem is obtained by locally constructing reduced order models for the solution space
on coarse regions and then employing a global formulation on the resulting reduced space.

The Brinkman model can be written as
\begin{eqnarray*}
 \nabla p -\mu\Delta {u} +\kappa^{-1} {u} = &{f} \quad \mbox{ in } \Om,&\\
\mbox{div}\,{u}= &0\quad \mbox{ in } \Om,&
\end{eqnarray*}
where $p$ is the fluid pressure and $u$ represents the velocity. Here, $\mu$ is the viscocity and
$\kappa=\kappa(x)$ is a heterogeneous multiscale coefficient that models the permeability
of the porous medium. We assume that the variations of $\kappa$ occur within a very fine
scale and therefore a direct simulation of this model is costly. As mentioned above,
one of the main advantages of the Brinkman model is that it can capture Stokes and Darcy type flow behavior depending
on the value of $\kappa$ without the usage of a complex interface condition as needed in the Stokes-Darcy interface model. This is very
convenient when modeling complicated porous scenarios such as a vuggy medium.
However, this advantage of the Brinkman model does not come for free: it brings the challenge of effectively designing
numerical homogenization or upscaling methodologies since the resulting upscaling
method must capture the correct flow behavior in corresponding regions. This difficulty
increases in the case of  high-contrast coefficients due to the fact that, in
a single coarse region, the permeability field can have variations of several order of magnitude that make it difficult
to compute effective parameters  for the permeability or boundary conditions using classical
multiscale finite element methods.
In this work, we develop an efficient (multiscale) solver based on the GMsFEM framework \cite{egh12} for the
Brinkman flow in heterogeneous high-contrast permeability fields. In this framework, as in many
other multiscale model reduction techniques, one divides the computation into two stages, i.e.,
the offline stage and the online stage. In the offline stage, a reduced dimension space is constructed,
and it is then used in the online stage to construct multiscale basis functions. These multiscale basis
functions can be re-used for any input parameter to solve the problem on a coarse grid. The main
idea behind the construction of the offline and online spaces is to design appropriate local
spectral-based selection of important modes that generate the snapshot space. In \cite{egh12},
several general strategies for designing the local spectrum-based selection procedures were proposed.
In this work, we focus on the generation of snapshots spaces, and rigorous convergence analysis of the resulting
coarse approximation. Further, we establish stability estimate of the mixed GMsFEM (in the form of inf-sup
conditions) for the proposed reduced dimension spaces. The convergence analysis extends that for elliptic equations with high-contrast coefficients \cite{egw10}.

We present several numerical examples to illustrate the performance of the proposed approach.
In particular, four different high-contrast multiscale permeability fields, which are
representative of Brinkman flow scenarios: Darcy flow in high-contrast regions composed of channels
and inclusions, Darcy flow in high-contrast regions composed of background, Stokes flow in
high-contrast regions composed of channels and inclusions, and Stokes flow in high-contrast
regions composed of background. All the numerical results indicate that the proposed GMsFEM
is robust and accurate.

The rest of the paper is organized as follows. In Section \ref{prelim}, we present
preliminaries on the Brinkman model and the GMsFEM. The construction of the coarse
spaces for the GMsFEM is displayed in Section \ref{sec:off}.
In Section \ref{sec:numerical}, numerical results for several representative examples are showed.
The proofs of our main results, including stability and a priori error estimates, are exhibited in Section \ref{Analysis}.
Finally, we conclude our paper with some remarks in Section \ref{sec:conclusion}. 
\section{Preliminaries}
\label{prelim}
Now we describe the Brinkman model in a more detailed manner. Let $\Omega$ be a polygonal domain in $\mathbb{R}^d$
($d = 2, 3$) with a boundary $\partial\Omega$. Then the Brinkman model reads: find $({u},p)\in
 {H}^{1}(\Om))^d\times L^{2}_{0}(\Om)$) such that
\begin{subequations}\label{eq:Brinkman.fs}
\begin{align}
\label{eq:BrinkmanMomentum.fs}
 \nabla p-\mu\Delta {u}+\kappa^{-1} {u} = &{f}\quad \mbox{ in } \Om,&\\
\label{eq:BrinkmanDiv.fs}
\mbox{div}\,{u}= &0\quad\mbox{ in } \Om,&\\
\label{eq:Brinkman.bc} {u}= &{g}\quad\, \mbox{ on } \partial \Om.&
\end{align}
\end{subequations}
Here the source term $f \in (L^2(\Omega))^d$, the boundary condition
${g} \in ({H}^{\frac12}(\partial \Omega))^d$, and $\kappa^{-1}$ is a positive
definite heterogeneous tensor field with high-contrast. Without loss of
generality, we assume the viscosity parameter $\mu=1$ and $g=0$ throughout.

To simplify the notation, we denote by ${V}(\Om)=(H_0^1(\Omega))^d$ and $W(\Om)=L^2_0(\Omega)$.
The variational formulation of the problem is given by: find $u\in {V}(\Om)$ and $p\in W(\Om)$ such that
\begin{alignat*}{3}
a(u, v) +& b(v, p) &=& l_f(v) &\qquad  &\mbox{for all } v \in V(\Omega), \\
&b(u,q) &=& 0&  \qquad&\mbox{for all } q \in W(\Omega),
\end{alignat*}
where the bilinear forms $a(\cdot,\cdot)$ and $b(\cdot,\cdot)$ are respectively defined by
\begin{alignat*}{3}
 a(u,v)&=&\avrg{\nabla u,\nabla v}{\Om}&+\avrg{\kappa^{-1}u,v}{\Om},& \mbox{for all } u,v\in V(\Om), \\
b(u,p)&=&\avrg{\DIV u,p}{\Om},&  &\mbox{ for all } u\in V(\Om),\ p\in W(\Om),
\end{alignat*}
and the linear form $l_f$ is given by
\[l_f(v) =\avrg{f,v}{\Om}, \mbox{ for all } v\in V(\Om),\]
where $\avrg{\cdot\,,\,\cdot}{\Om}$ denotes the $L^2$ inner product over $\Om$.

Let $\mathcal{T}_H$ be a coarse-grid partition of the domain $\Omega$ and $\mathcal{T}_h$ be a
conforming fine triangulation of $\Omega$. We assume that $\mathcal{T}_h$ is a
refinement of $\mathcal{T}_H$, where $h$ and $H$ represent the mesh size of a fine and coarse cell
respectively. Typically we assume that $0 < h \ll H < 1$, and that the triangulation $\mathcal{T}_h$
is fine enough to fully resolve the spatial variations of the coefficient $\kappa$ while $H$ is too
coarse to accurately resolve this spatial variations inside a coarse element, and the coefficient
$\kappa$ may have large variations within the coarse block.
On the triangulation $\mathcal{T}_h$,
we introduce the following finite element spaces 
\begin{align*}
{V}_h  &:= \{ {v} \in V(\Omega)| {v}|_K \in (P^2(K))^d
\mbox{ for all } K \in \mathcal{T}_h \}, \\
W_h &:= \{ q \in W(\Omega)| w|_K \in P^0(K), \mbox{ for all } K \in \mathcal{T}_h\}.
\end{align*}

The standard mixed finite element method for problem \eqref{eq:Brinkman.fs} is to
seek an approximation $({u}_h, p_h)$ in the finite element space
${V}_h \times W_h \subset {V}(\Omega) \times W(\Omega)$ such that
\begin{alignat*}{3}
a(u_h, v) +& b(v, p_h) &=& l_f(v) &\qquad &\mbox{for all $v \in V_h, $} \\
&b(q, u_h) &=& 0 &&\mbox{for all $q \in W_h$},
\end{alignat*}
or which is equivalent to the solution of the following linear system
\begin{align*}
 \begin{pmatrix}
  A&B\\
B^{T}&0
 \end{pmatrix}
\begin{pmatrix}
  u\\p
 \end{pmatrix}
=\begin{pmatrix}
  F\\ 0
 \end{pmatrix}.
\end{align*}
Here the matrices denote
\begin{eqnarray}
   &v^T A u = a(u,v),& \mbox{ for all } u,v\in V_h,\\
   &q^TB u = b(u,q),& \mbox{ for all } u\in V_h \mbox{ and } q\in W_h.
\end{eqnarray}
Note that here and below, in order to simplify notation, we are using the same notation for finite element functions
and their corresponding vector representations.

It is well known the mixed  finite element formulation described above is stable; see for instance \cite{XXX08}. In the case of high-contrast media,
a very refined grid is needed in order to fully resolve small scale features, and
thus it is prohibitively expensive to solve the resulting system. Meanwhile, if we naively apply
$P^2/P^0$ finite element spaces over the coarse mesh $\mathcal{T}_H$, the resulting
system is small but obviously the solution can only represent a poor approximation to the
exact solution. To turn around the dilemma, we follow the GMsFEM framework proposed
in \cite{egh12}.

In the GMsFEM methodology one divides the computations into onffline and online computations.
The offline computations are based upon a preliminary dimension reduction of the
fine-grid finite element spaces (that may include
dealing with additionally important physical parameters, uncertainties and nonlinearities), and then the online procedure (if needed)
is applied to construct a reduced order model on the offline space. We start by constructing  offline spaces.

We construct the coarse function space
$${V}^{\text{off}} := \mbox{span}\{{\phi}_i\}_{i=1}^{N_c},$$
where $N_c$ is the number of coarse basis functions. Each ${\phi}_i$ is supported in some coarse
neighborhood $w_l$. For the pressure field $p$, we use the space of piecewise constant functions over the coarse triangulation $\mathcal{T}_H$, that is,
\begin{equation}\label{eq:def:Woff}
W^{\text{off}} := \{ q \in L^2_0 (\Omega)| q|_K \in P^0(K), \; \mbox{ for all } \; K \in \mathcal{T}_H\}.
\end{equation}
We denote $N_H=\dim W^{\text{off}}$.

The idea is then to work on the reduced spaces $V^\text{off}\times W^\text{off}$ instead of the
original spaces $V(\Om)\times W(\Om)$.
In the general GMsFEM methodology, these offline spaces are used  in the online
computations where a further reduction may be  performed; see  \cite{egw10,egh12} for details.
The overall performance of the resulting GMsFEM depends on the
approximation properties of the resulting offline and online coarse spaces.
In this paper we focus on the construction of the offline spaces only. We mention that this is sufficient
for the effective numerical upscaling of the Brinkman model proposed above where neither
parameters or nonlinearities are considered. The more general case with additional parameters can also
be studied using the proposed method, but it requires online dimension reduction (\cite{egw10,egh12,egt11})
and thus defer to a future study.

The GMsFEM seeks an approximation $({u}_0, p_0)\in{V}^{\text{off}} \times W^{\text{off}}$ which satisfies the coarse scale offline formulation,
\begin{subequations}\label{eq:coarse_system}
\begin{alignat}{3}
a(u_0, v) + &b^{t}(p_0, v) &=& l_f(v) \, &\mbox{ for all $v \in V^{\text{off}}$}, \\
&b(u_0, q) &=& 0 &\mbox{for all $q \in W^{\text{off}}$}.
\end{alignat}
\end{subequations}
We can interpret the method in the following way using  matrix representations. Recall that both coarse basis functions
$\{{\phi}_i\}^{N_c}_{i=1}$ and $\{q_i\}^{N_H}_{i=1}$ are defined on the fine grid, and
can be represented by the fine grid basis functions. Specifically, we introduce the following matrices:
\[
R^T_0 = [{\phi}_1, \dots, {\phi}_{N_c}], \quad Q^T_0 = [q_1, \dots, q_{N_H}],
\]
where we identify the basis $\phi_i$ and $ q_i$ with their coefficient vectors in the fine grid basis.
Then the matrix analogue of the system \eqref{eq:coarse_system} can be equivalently written as
\begin{align}\label{eqn:offlineSystem}
 \begin{pmatrix}
  R_0AR_0^{T}&R_0B Q_0^T\\
Q_0 B^{T}R_0^{T}&0
 \end{pmatrix}
\begin{pmatrix}
  u_0\\p_0
 \end{pmatrix}
=\begin{pmatrix}
  R_0F\\0
 \end{pmatrix}.
\end{align}
Further, once we solve the coarse system \eqref{eqn:offlineSystem}, we can recover the fine scale solution by $R_0^T u_0$.
In other words, $R^{T}_{0}$ can be regarded as the transformation (also known as interpolation, extension, and
downscaling) matrix  from the space ${V}^{\text{off}}$ to the space ${V}^{h}$.

The accuracy of the GMsFEM relies crucially on the coarse basis functions $\{\phi_i\}$.
We shall present one novel construction of suitable basis functions for the Brinkman equation in Section \ref{sec:off}.
%
\section{The construction of the space ${V}^{\text{off}}$}
\label{sec:off}
In this section, we present the construction of the space ${V}^{\text{off}}$ in detail. For the
pressure field $p$, we simply use piecewise constant functions over the coarse grid as
defined in (\ref{eq:def:Woff}). Therefore the focus below is on the construction of the offline
velocity space $V^\text{off}$. To this end, we first introduce the concept of (harmonic)
extension of boundary data in the Brinkman sense,
which will play an important role in the construction.
The precise definition is given below.
\begin{Definition}[Brinkman Extension]
\label{Definition:Brinkman_ext}
For a domain $D \subset \mathbb{R}^d$, 
we define the Brinkman extension of any ${v}\in (H^{\frac12}(\partial D))^d$, denoted by
$\mathcal{H}({v}) \in ({H}^1(D))^d$, to be the unique solution of the following homogeneous
Brinkman equation (with $|D|$ being the measure of $D$)
\begin{alignat*}{3}
 \nabla p-\mu\Delta \mathcal{H}({v})+&\kappa^{-1} \mathcal{H}({v})&=& 0& \mbox{ in } D,\\
  &\mbox{div} \, \mathcal{H}({v})&=& \frac{1}{|D|} \int_{\partial D} {v} \cdot {n} & \mbox{ in } D,\\
 & \quad\;\mathcal{H}({v})&=& {v}& \mbox{ on } \partial D.
\end{alignat*}
\end{Definition}

\begin{remark}
In practice, the extension $\mathcal{H}({v})$ is the numerical solution of the equation in the
fine-scale finite element space ${V}_h(D) \times W_h(D)$, where $D$ is a coarse block (see Fig.
\eqref{schematic} for an illustration of coarse block and coarse neighborhood). This computation
can be efficiently performed due to the moderated size of the coarse regions. Besides, the computations
can be carried out in parallel, if the computations are required over all coarse regions.
\end{remark}
Now we are ready to state the detailed construction of the offline velocity space $V^\text{off}$.
Our construction consists of the following three steps. We defer the
analysis of the resulting GMsFEM method to Section 5.
\paragraph{Step 1: Building multiscale partition of unity functions.}
First we introduce a set of generalized global partition of unity functions on the coarse grid.
We denote the set of all coarse edges by $\mathcal{E}_H$, and consider the following
finite element space:
\[
  M_H := \{ {v} \in C^0(\mathcal{E}_H):\ \ {v}|_F \in P^2(F), \quad \mbox{ for all } F \in \mathcal{E}_H\}.
\]
Let $\mathcal{P}_H$ be the set of the shape functions of the space $M_H$. Then $\mathcal{P}_H$
also forms a partition of unity functions over the skeleton $\mathcal{E}_H$.

Next we introduce the set of multicale partition of unity functions for a two-dimensional domain $\Omega$.
We remark that the construction for the three-dimensional case is similar. For
any $\chi \in \mathcal{P}_H$, let $\omega$ denote the support of $\chi$, and we call $\omega$ a
coarse neighborhood associated with $\chi$. In Figure \ref{schematic}, we sketch all three possible
types of the coarse neighborhood, $\omega_1$, $\omega_2$ and $\omega_3$, respectively. $\omega_1$ corresponds
to partition of unity funtion $\chi$ having nodal value 1 at the coarse node $i$; $\omega_2$ represents the support of $\chi$
 valuing 1 at node $j$, and $\omega_3$ stands for support of $\chi$ equaling to 1 at node $k$.

\begin{figure}[htb!]
  \centering
  \includegraphics[width=0.65 \textwidth]{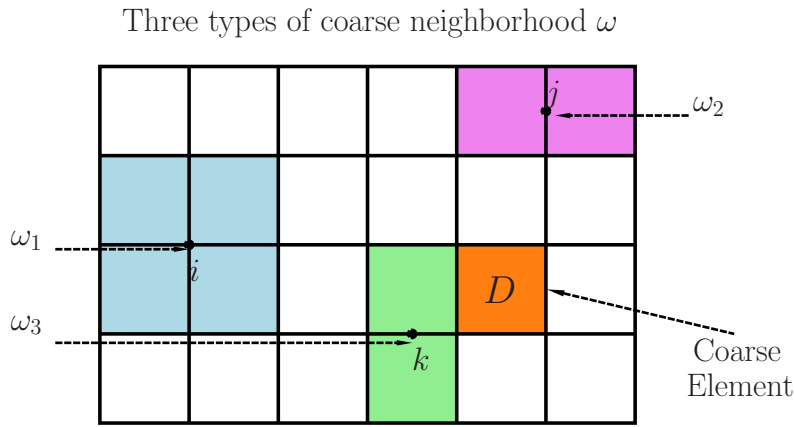}
  \caption{Illustration of three types of coarse neighborhoods and coarse element:
  $\omega_1$, $\omega_2$ and $\omega_3$ denote the support of partition of unity functions $\chi$.}
  \label{schematic}
\end{figure}
For each $\chi \in \mathcal{P}_H$, we have two Brinkman extensions of $\chi$: $\mathcal{H}({\chi}^x)
\text{ and } \mathcal{H}({\chi}^y)$. Here $\mathcal{H}({\chi}^x)$ is the Brinkman extension of the
Dirichlet data ${\chi}^x = (\chi, 0)$, and it is defined on each coarse block $D \subset
\omega$. The extension $\mathcal{H}({\chi}^y)$ is defined similarly, with ${\chi}^y = (0,\chi)$.
We note that these vector functions can be extended by $0$ to the entire domain $\Omega$,
since $\chi$ vanishes over $\partial \omega$. Finally, the generated partition of unity
functions associated with $\chi$ is $\chi_x = \frac{1}{2} (\mathcal{H}({\chi}^x))_x$,
$\chi_y = \frac{1}{2} (\mathcal{H}({\chi}^y))_y$, where $(\cdot)_x $ and $(\cdot)_y$ denote
the first component and second component of a vector, respectively. Thus, for each
$\chi \in \mathcal{P}_H$, we generate two partition of unity functions supported on $\omega$.
All these extensions together form a set of multiscale partition of unity functions,
which are denoted by:
\[
\mathcal{P}_{\text{ext}} = \{\chi_i\}^{N_p}_{i = 1},
\]
where $N_p$ is the number of multiscale partition of unity functions. We note that the set $\mathcal{P}_\text{ext}$
does not have the default property of partition of unity over the domain $\Omega$, but only over the skeleton $\mathcal{E}_H$, i.e.,
\begin{equation}\label{eq: partition}
\sum_{i=1}^{N_p} \chi_i \equiv 1 \quad \text{on $\mathcal{E}_H$.}
\end{equation}

\paragraph{Step 2: Constructing local snapshot space ${V}^{\omega}_{\text{snap}}$.}
In this step, we construct the local snapshot spaces. Proceeding as before, for each $\chi \in \mathcal{P}_{ext}$,
we let $\omega$ denote its support set, and call it the coarse neighborhood associated with $\chi$,
which consists of either two or four coarse blocks, cf. Figure \ref{schematic}. The construction of
the local snapshots is further divided into two substeps: generating the snapshot space over all coarse
neighborhoods $\omega$ and then improving their linear independence.

\noindent{\bf Step 2.1. Computing snapshots:} For each coarse neighborhood $\omega$, let
$\textsl{J}_{h}(\partial \omega)$ denote the set of fine nodes on $\partial \omega$. Let
$\delta_k \in C^0(\partial \omega)$ denote the shape function associated with the node
$x_k \in \textsl{J}_h(\partial \omega)$. i.e., $\delta_k \in C^0(\partial \omega)$ is the
piecewise linear function that takes value $1$ on the node $x_k$ and vanishes on all other
nodes. For each $\delta_k$, it generates two Brinkman extensions:
\begin{equation*}
{\psi}_{k,x} = \mathcal{H}((\delta_k, 0))\quad \mbox{and}\quad {\psi}_{k,y} = \mathcal{H}((0, \delta_k)).
\end{equation*}
Now the raw snapshot space on $\omega$ is given by
\[
\widetilde{{V}}^{\omega}_{\text{snap}} = \text{span}\{ {\psi}_{i,x}, {\psi}_{i,y}: \; \text{for all} \; x_i \in \textsl{J}_{h}(\partial \omega)\} + \text{span}\{(1, 0), (0, 1)\}.
\]
Here we artificially add two constant vectors in the basis; see Remark \ref{rmk:constant_vector}
below for the discussions.

\noindent{\bf Step 2.2: Improving linear independence of snapshots:}
After obtaining a family of local functions for each
coarse neighborhood, we need to discard the possible linearly dependent local snapshots.
To this end, we use a spectral problem based on the Euclidian inner product. Specifically,
Let $U$ be a matrix with columns being the local snapshots vector representation. We extract
the important modes of $U$ through the spectral decomposition of $U^TU$. In this manner, we keep the
linearly independent snapshots for each coarse neighborhood $\omega$ and denote the resulting space by
\[
{V}_{\text{snap}}^{\omega} = \text{span}\{ {\psi}_{l}^{\omega, \text{snap}}: 1\leq l \leq L^{\omega} \},
\]
with $L^{\omega}$ being the number of local basis functions for the coarse neighborhood $\omega$.

\paragraph{Step 3: Building the offline space ${V}^{\text{off}}$.}
In this final step, we build the global offline space ${V}^{\text{off}}$ from
the snapshot spaces  ${V}^{\omega}_{\text{snap}}$, and it involves two substeps:
constructing local offline space and constructing global offline space.

\noindent{\bf Step 3.1: Local multiscale space $\widetilde{{V}}^{\omega}_{\text{off}}$.}
The idea at this step is to extract only important information from the computed local
snapshots $V_\text{snap}^\omega$ corresponding to each coarse neighborhood $\omega$.
This can be achieved by performing a dimension reduction procedure in the space ${V}_{\text{snap}}^{\omega}$.
Namely, we consider the following spectral eigenvalue problem:
\begin{align}\label{eqn:eig}
A\widehat{{\Psi}}_k^{\omega,\text{off}} =\lambda_k S \widehat{{\Psi}}_k^{\omega,\text{off}},
\end{align}
where the matrices $A$ and $S$ are defined by
\begin{equation*}
  \begin{aligned}
    \displaystyle A &= [a_{mn}] = \int_{\omega} \kappa(x) \nabla {\psi}_m^{\omega,\text{snap}} \cdot \nabla {\psi}_n^{\omega,\text{snap}},\ 1 \le m, n \le L^{\omega}\\
  \displaystyle S &= [s_{mn}] = \int_{\omega}\kappa(x) {\psi}_m^{\omega,\text{snap}} {\psi}_n^{\omega,\text{snap}}, \ 1 \le m, n \le L^{\omega}.
   \end{aligned}
 \end{equation*}
Then we reorder the eigenvalues $\lambda_k$ are in an ascending order, and denote $\widehat{{\Psi}}_k^{\omega,\text{off}}$
as the coresponding eigenvectors.

To generate the offline space, we then choose the $M_{\text{off}}$ smallest eigenvalues of
\eqref{eqn:eig} and the corresponding eigenvectors in the respective space of snapshots by
setting $\widetilde{{\Psi}}_k^{\omega,\text{off}} = \sum_j \widehat{\Psi}_{kj}^{\text{off}}
\psi_j^{\omega,\text{snap}}$, where $\widehat{\Psi}_{kj}^{\text{off}}$ are the coordinates of
the vector $\widehat{{\Psi}}_k^{\omega,\text{off}}$. We then construct the offline space
$ \widetilde{{V}}_{\text{off}}^{\omega}$ corresponding  to the coarse neighborhood $\omega$ as
 $$
 \widetilde{{V}}_{\text{off}}^{\omega} = \text{span}\left( \widetilde{{\Psi}}_{1}^{\omega,\text{,off}}, \ldots,  \widetilde{{\Psi}}_{M_{\text{off}}}^{\omega,\text{off}}\right).
$$

We note that this step is performed only on each coarse neighborhood $\omega$. The dimensionality of
the space $\widetilde{{V}}^{\omega}_{\text{off}}$ solely depends on the eigenvalue problem
\eqref{eqn:eig} within the neighborhood $\omega$. It is known that this
space is related to important features of the media (cf. \cite{egw10}) such as high-conductivity
channels and inclusions, and thus its dimensionality depends on the structure of the heterogeneities.

\begin{remark}\label{rmk:constant_vector}
In the construction of the local snapshot space,
we have added constant functions in addition to spectral basis functions. Hence, the constant function,
which is the eigenvectors corresponding to the zero eigenvalue of \eqref{eqn:eig},
will always be in the offline space. By the construction of the offline space, each offline space
contains the partition of unity functions, and the smallest offline space consists of those
partition of unity functions only. This will be crucial in the stability analysis of the methods in Section 5.
\end{remark}

\noindent{\bf Step 3.2: Construction of the global offline space ${V}^{\text{off}}$ by partition
of unity.} The local multiscale spaces $\widetilde{{V}}^{\omega}_{\text{off}}$ are defined only on
each neighborhood $\omega$. However, it is not conforming if we simply extend the functions by $0$ to the
whole domain. We obtain a global conforming offline space $V^\text{off}$ as follows.

First, we multiply each local offline space $\widetilde{{V}}^{\omega}_{\text{off}}$ by the
corresponding partition of unity function $\chi$:
\[
\chi \widetilde{{V}}^{\omega}_{\text{off}} = \text{span}\left( \chi \widetilde{{\Psi}}_{1}^{\omega,\text{off}}, \ldots,  \chi \widetilde{{\Psi}}_{M_{\text{off}}}^{\omega,\text{off}}\right).
\]
Then the space $\chi \widetilde{{V}}^{\omega}_{\text{off}} \subset {H}^1_0(\omega)$, and we can extend
the functions in $\chi\widetilde{V}^\omega_\text{off}$ to the whole domain $\Omega$ by zero,
which is still denoted as $\chi \widetilde{{V}}^{\omega}_{\text{off}}$. Finally, we need to make
a correction of the divergence of the resulting functions to satisfy the following condition:\[
\nabla \cdot {V}^{\text{off}} \subset W^{\text{off}}.
\]
To this end, for each basis function $\chi \widetilde{{\Psi}}_i^{\omega, \text{off}}$, within each
coarse block $D \subset \omega$, we keep its trace along $\partial D$ and modify its interior values
to be the Brinkman extension $\mathcal{H}(\chi \widetilde{{\Psi}}_i^{\omega, off}|_{\partial D})$.
We denote this modified space by $\mathcal{H}(\chi \widetilde{{V}}^{\omega}_{\text{off}})$. The global
offline space $V^\text{off}$ results from assembling  all these modified local spaces as:
\[
{V}^{\text{off}}:= \{{v} \in ({H}^1_0(\Omega))^d: {v}|_{\omega} \in \mathcal{H}(\chi \widetilde{{V}}^{\omega}_{\text{off}})\}.
\]
This completes the construction of the offline space ${V}^{\text{off}}$. Finally, we refer to Section
\ref{prelim} for the coupling of the offline basis functions.

\section{Numerical results}\label{sec:numerical}
Now we test our framework with several examples. In our experiments, we take the domain
$\Omega = [0, 1] \times [0, 1]$, the source term $f=0$, and the boundary condition
is the constant horizontal velocity:
\[
{g} = (1, 0) \quad \mbox{on} \quad \partial \Omega.
\]
We study the model with different (inverse) permeability fields $\kappa^{-1}$
depicted in Figure \ref{fig:perms}. Figure \ref{fig:perms}(a) shows a fast Darcy
flow going through the slower region; in Figure \ref{fig:perms}(b), we exposit a slower Darcy flow past Darcy flow regions;
 in Figure \ref{fig:perms}(c), a
free flow going across the Darcy flow region is represented; and in Figure \ref{fig:perms}(d),
a Darcy flow passing the strong free flow region is shown.

\begin{figure}[hbt!]
\centering
\begin{tabular}{cc}
\includegraphics[trim = 1cm 0cm .5cm 0cm, clip=true,width = 0.45\textwidth]{{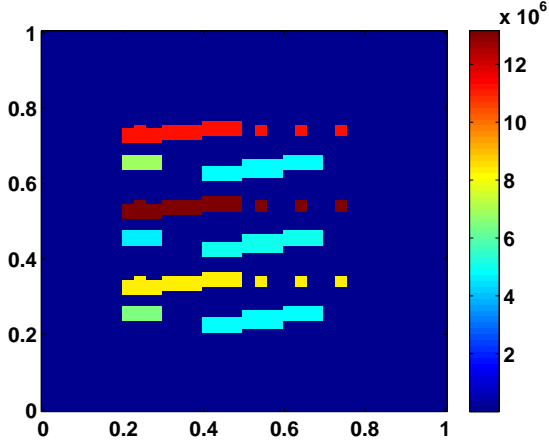}} & \includegraphics[trim = 1cm 0cm .5cm 0cm, clip=true,width = 0.45\textwidth]{{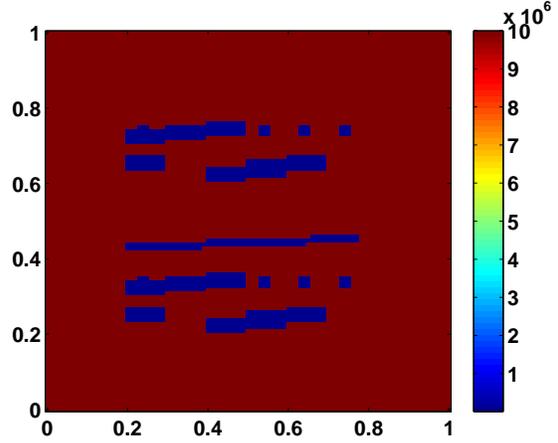}}\\
\small{(a) Fast Darcy flow going through a slower region.} & \small{(b) A slower Darcy flow past Darcy flow regions.}\\
\includegraphics[trim = 1cm 0cm .5cm 0cm, clip=true,width = 0.45\textwidth]{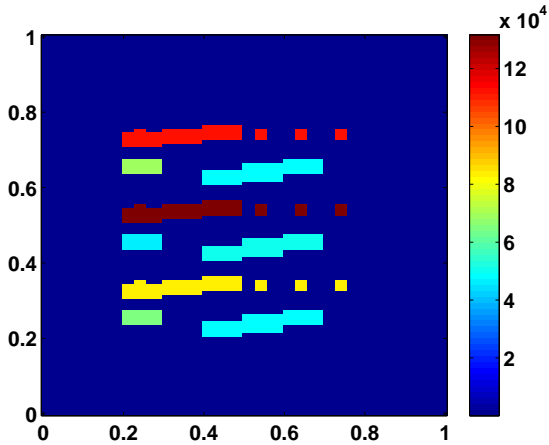}&\includegraphics[trim = 1cm 0cm .5cm 0cm, clip=true,width = 0.45\textwidth]{{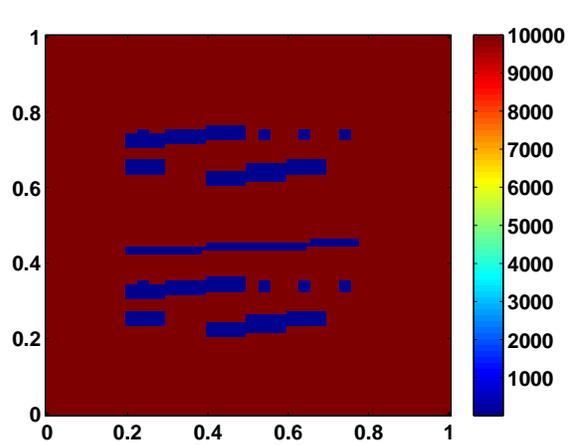}}\\
\small{(c) A free flow going across the Darcy flow region.}&\small{(d) A Darcy flow passing the strong free flow region.}
\end{tabular}
\caption{Four representative inverse permeability fields $\kappa^{-1}$. }\label{fig:perms}
\end{figure}

We divide the computational domain $\Omega=[0\;1]\times[0\;1]$ into $N_f=1/h^2$ equal squares (where each square is
further divided into two triangles), and use $P_2/P_0$ elements on the fine mesh with $h=1/100$. We
use a coarse-mesh size $H=1/10$ where we divide the domain $\Omega=[0,1]\times[0,1]$ into $1/H^2$ squares.

We depict the fine-scale solution, and three coarse-scale solutions with coarse spaces of dimensions
798, 1110 and 2726 in Figure \ref{fig:solns_HCC}. The dimension of the fine scale velocity space
$V_h$ is 80802. In these numerical tests, we use the value of the permeability field
$\kappa^{-1}$ from Figure \ref{fig:perms}(a). We observe that a larger coarse space yields a
better approximation of the fine-scale solution. Further, we have the following observations.
\begin{itemize}
\item[(a)] The use of one single basis function for each node gives large errors
and thus it is necessary to add spectral basis functions.

\item[(b)] The error decreases as more spectral basis functions are added in each
coarse-grid block.

\item[(c)] The error decreases if the solution displays fast flow in some regions
instead of Darcy flow over the whole region under the same contrast.
\end{itemize}

\begin{figure}[htb!]\centering
\begin{tabular}{cc}
\includegraphics[height = 5cm,width=7.5cm]{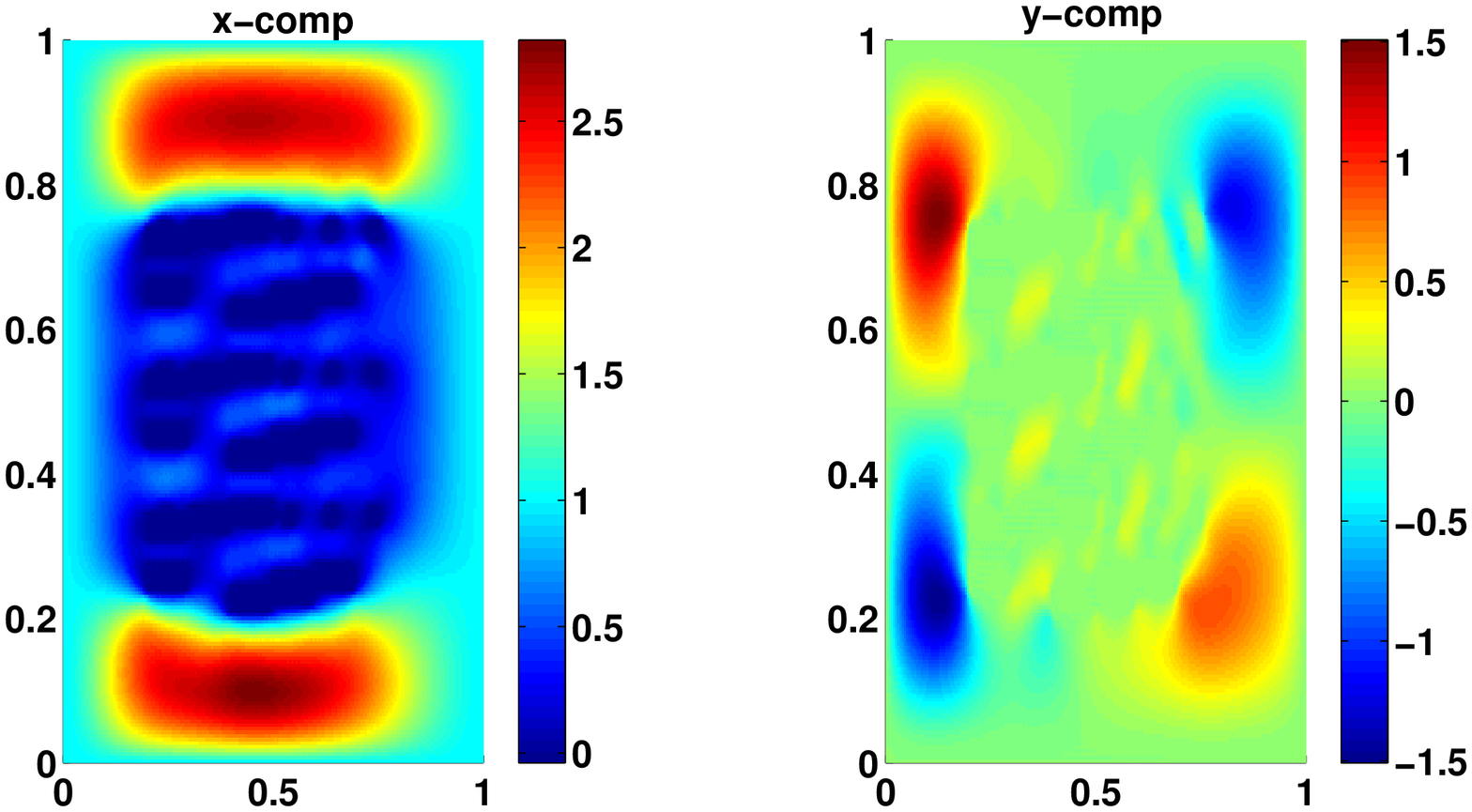}&\includegraphics[height = 5cm,width=7.5cm]{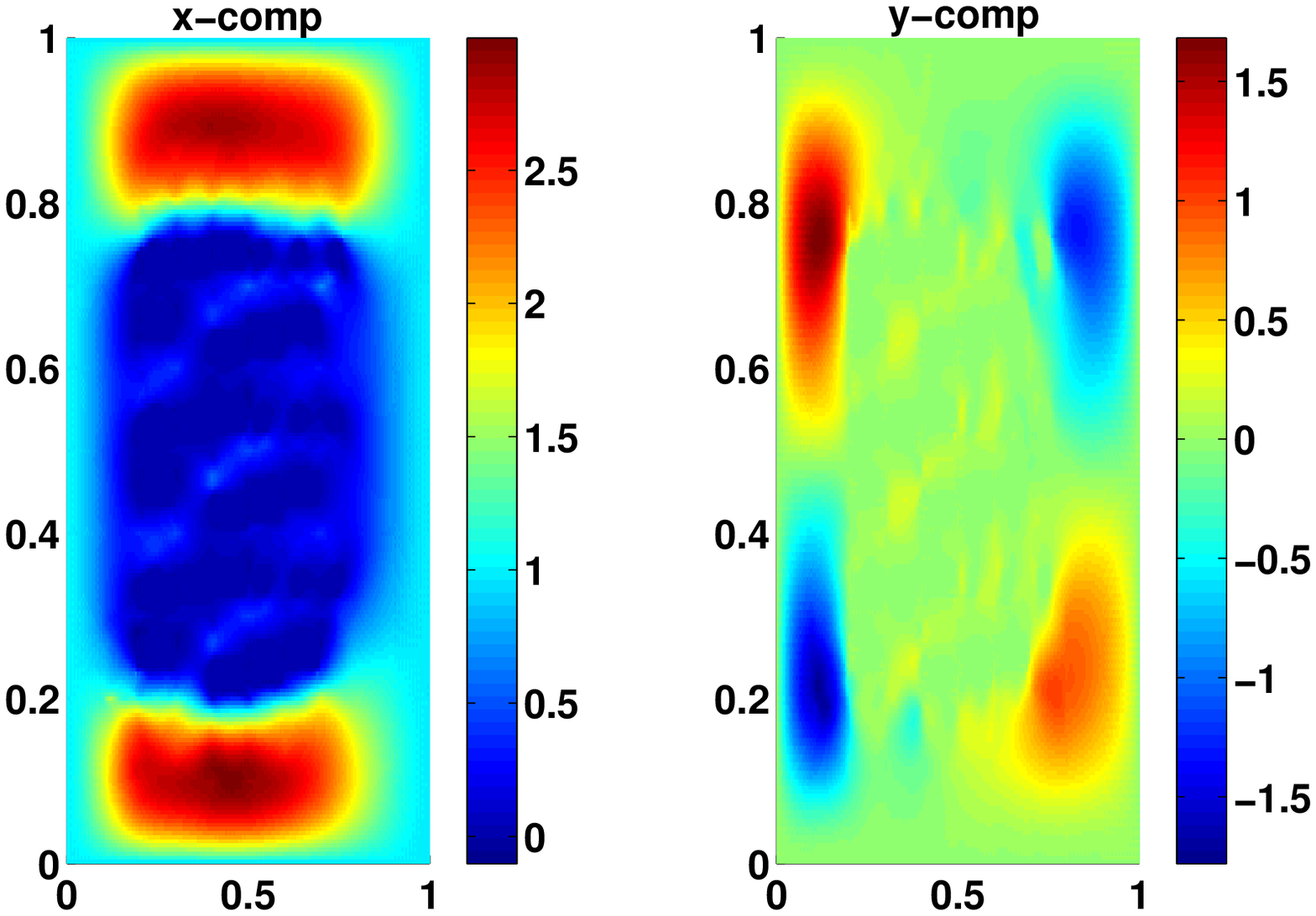}\\
 \small{(a) Fine-scale solution.}&\small{(b) Coarse solution with solution space of 798.}\\
\includegraphics[height = 5cm,width=7.5cm]{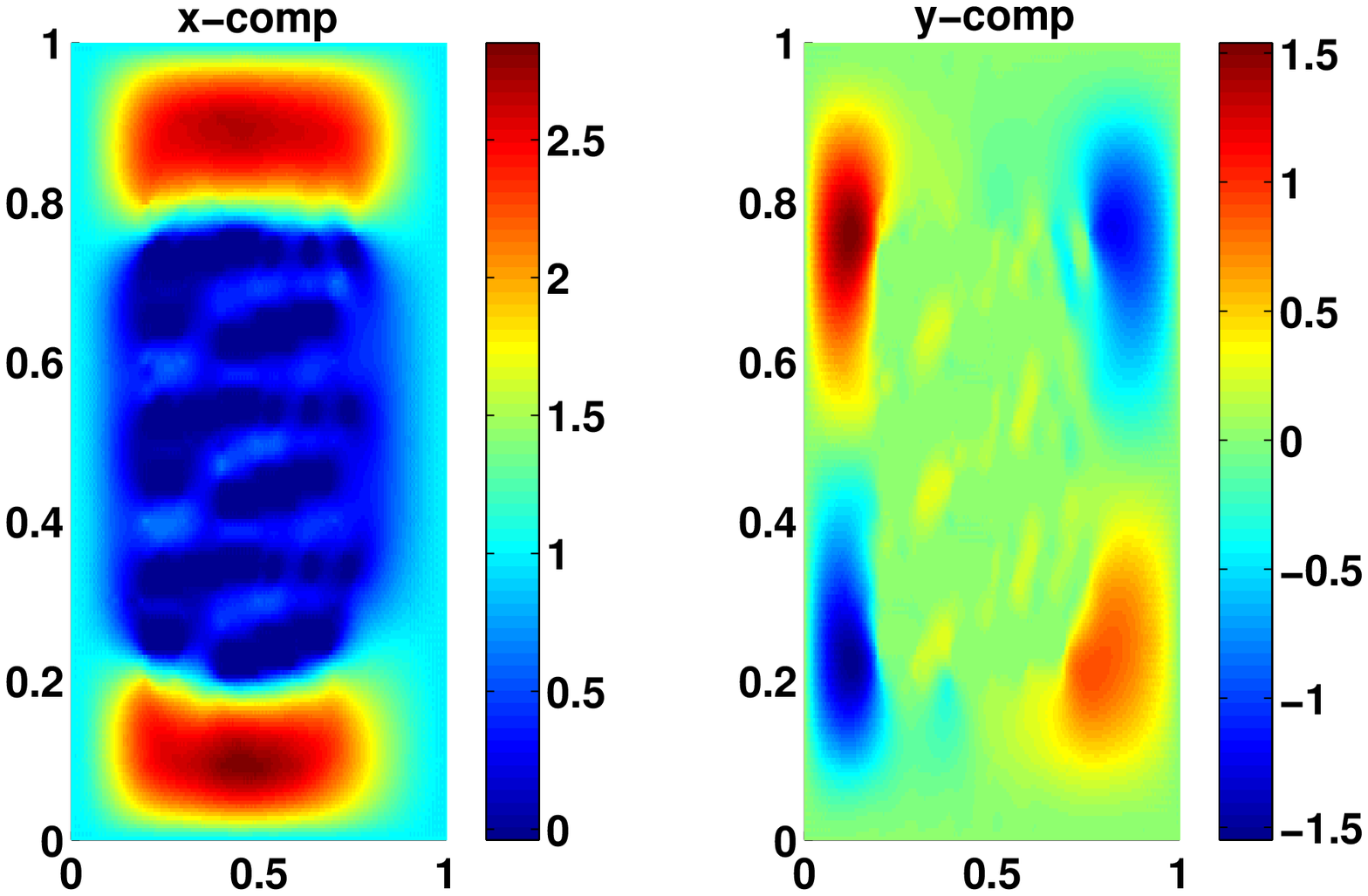}&\includegraphics[height = 5cm,width=7.5cm]{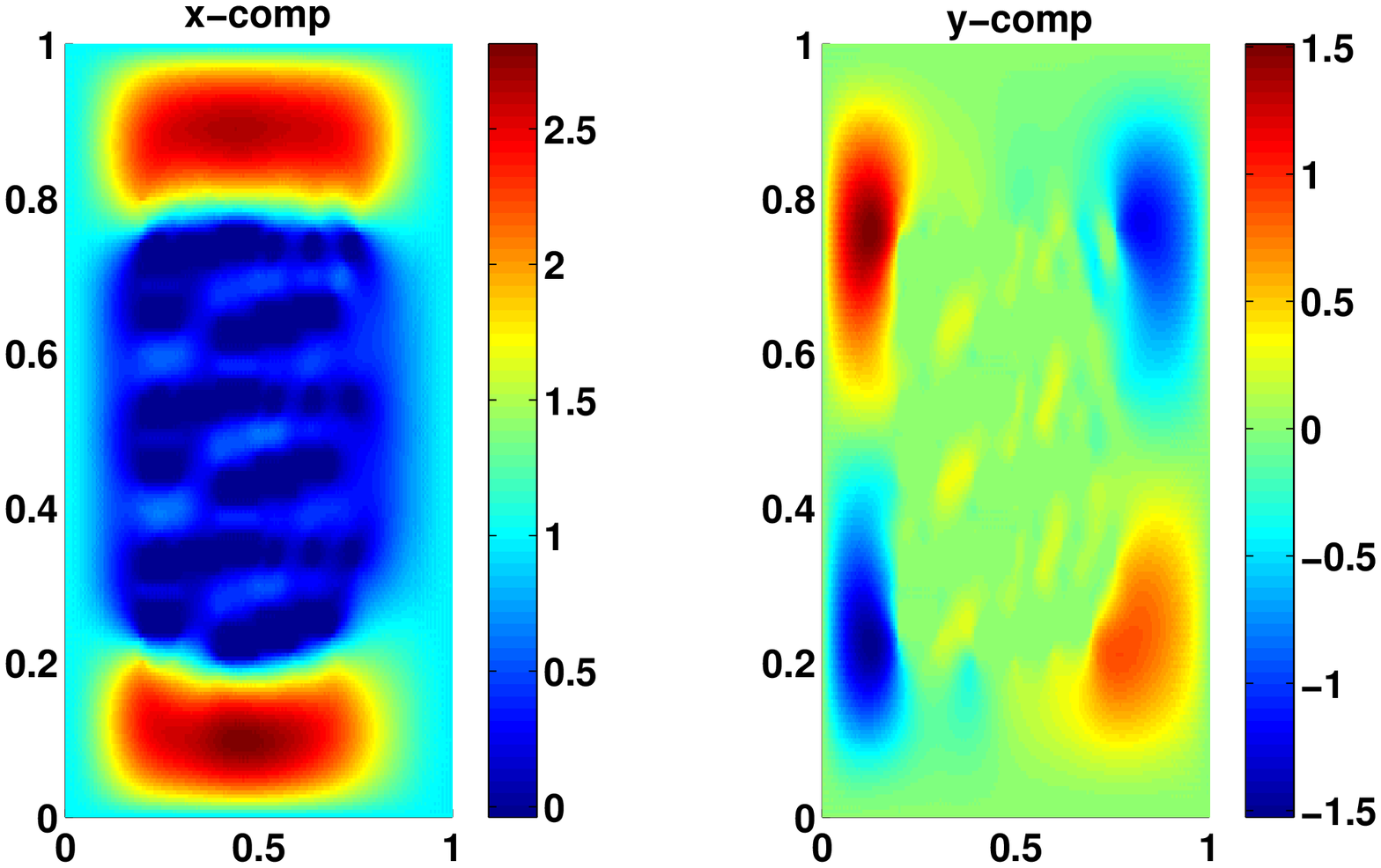}\\
 \small{(c) Coarse solution with solution space of 1110.}&\small{(d) Coarse solution with solution space of 2726.}
 \end{tabular}
   \caption{The fine-scale solution and three coarse-scale solutions with different dimensions of coarse spaces
   using the permeability field $\kappa^{-1}$ in Figure \ref{fig:perms} (a). }\label{fig:solns_HCC}
\end{figure}

In Tables \ref{table:HCCResult}-\ref{table:Reverse of HCC_stokes Result}, we present the results with the multiscale partition of unity
functions as required by the conforming Galerkin formulation corresponding to permeability fields
in Figure \ref{fig:perms}. In the tables, the first column shows the dimension of the offline space $V^\text{off}$,
and the $L^{2}$-weighted error between the offline solution $u^{\text{off}}$ and the fine-scale solution
$u$ and the $H^1$-weighted energy error are calculated respectively by
\begin{equation*}
  \|u-u^\text{off}\|_{L^2_\kappa(D)}=\frac{\normL{\kappa^{-1/2}(u-u^{\text{off}})}{\Om}}{\normL{\kappa^{-1/2}u}{\Om}}
  \quad \mbox{and}\quad
  \|u-u^\text{off}\|_{H^1_\kappa(D)}=\frac{\normL{\kappa^{-1/2}\nabla(u-u^{\text{off}})}{\Om}}{\normL{\kappa^{-1/2}\nabla u}{\Om}}.
\end{equation*}

\begin{table}[htb!]
\centering
\caption{Numerical results for problem \eqref{eq:Brinkman.fs} with $\kappa^{-1}$ in Figure \ref{fig:perms}(a). The $L^2$-weighted error and energy error are $66.34\%$ and $99.73\%$ for the MsFEM solution. In the simulation, the dimension of the snapshot space is fixed at 4498 with a weighted $L^2$ and energy relative error $1.26\%$ and $2.13\%$.}
 \label{table:HCCResult}
\begin{tabular}{|c|c|c|c|c|}
\hline
\multirow{2}{*}{$\text{dim}(V_{\text{off}})$}  &
\multicolumn{2}{c|}{  $\|u-u^{\text{off}} \|$ (\%) }  \\
\cline{2-3} {}&
$\hspace*{0.8cm}   L^{2}_\kappa(D)   \hspace*{0.8cm}$ &
$\hspace*{0.8cm}   H^{1}_\kappa(D)  \hspace*{0.8cm}$
\\
\hline\hline
       $888$     &       $35.46$    & $74.91$  \\
\hline
      $1372$    &    $26.62$    & $58.25$ \\
\hline
      $2028$    &    $11.79$    & $26.05$ \\
\hline
       $2204$   &   $8.61$  &$19.47$  \\

\hline
\end{tabular}
\end{table}

In Table \ref{table:HCCResult}, we display the velocity error results using a permeability field with
values of $\kappa^{-1}$ large in the background, with smalle inclusions values, cf. Figure
\ref{fig:perms}(a). For simplicity, we set a threshold value $\lambda^\text{off}$ for selecting
eigenvectors in the construction of the offline space. Specifically, for each coarse
neighborhood $\omega$, the offline space consists of those eigenvectors in Eqn \eqref{eqn:eig}
with eigenvalues $\lambda_k\geq \lambda^{\text{off}}$. Notice that the smaller is $\lambda^{\text{off}}$,
the larger is the velocity offline space. In the simulation, the choices $\lambda^{\text{off}}=1/3\mbox{,
}1/4\mbox{, }1/7\mbox{, and }1/10$ give the offline spaces of dimension 888, 1372, 2028 and 2204,
respectively. It is observed from Table \ref{table:HCCResult} that the error decreases from $74.91\%$ to $19.47\%$.

%
%
\begin{table}[htb!]
\centering
\caption{Numerical results for problem \eqref{eq:Brinkman.fs} with $\kappa^{-1}$ in Figure \ref{fig:perms}(b). The $L^2$-weighted error and energy error are $74.68\%$ and $130.42\%$ for the MsFEM solution. In the simulation, the dimension of the snapshot space is fixed at 4498 with a weighted $L^2$ and energy relative error $1.33\%$ and $13.03\%$. } \label{table:Reverse of HCCResult}
\begin{tabular}{|c|c|c|c|c|}
\hline
\multirow{2}{*}{$\text{dim}(V_{\text{off}})$}  &
\multicolumn{2}{c|}{  $\|u-u^{\text{off}} \|$ (\%) }  \\
\cline{2-3} {}&
$\hspace*{0.8cm}   L^{2}_\kappa(D)   \hspace*{0.8cm}$ &
$\hspace*{0.8cm}   H^{1}_\kappa(D)  \hspace*{0.8cm}$
\\
\hline\hline
       $682$     &       $7.86$    & $36.90$  \\
\hline
      $1512$    &    $1.85$    & $18.37$ \\
\hline
      $2230$    &    $1.51$    & $15.27$ \\
\hline
       $2744$   &   $1.38$  &$13.84$  \\

\hline
\end{tabular}
\end{table}%

The results in Table \ref{table:Reverse of HCCResult} are calculated with values of $\kappa^{-1}$
that are large
in inclusions, and small in the background, cf. Figure \ref{fig:perms}(b). Compared with results in
Table \ref{table:HCCResult}, the errors in Table \ref{table:Reverse of HCCResult} are slightly better
in the sense that the relative energy errors are smaller when using the same dimensional offline space.
In this numerical test, we take $\lambda^{\text{off}}=1/3\mbox{, }1/4\mbox{, }1/7\mbox{, and }1/10$
with the offline space of dimension 682, 1512, 2230 and 2744 respectively. From Table
\ref{table:Reverse of HCCResult}, the energy errors decrease from $36.90\%$ to $13.84\%$.

\begin{table}[htb!]
\centering
\caption{Numerical results for problem \eqref{eq:Brinkman.fs} with $\kappa^{-1}$ in Figure \ref{fig:perms}(c). The $L^2$-weighted error and energy error are $85.25\%$ and $73.85\%$ for the MsFEM solution. In the simulation, the dimension of the snapshot space is fixed at 4498 with a weighted $L^2$ and energy relative error $1.94\%$ and $3.54\%$.
 }
 \label{table:HCC-stokes-Result}
\begin{tabular}{|c|c|c|c|c|}
\hline
\multirow{2}{*}{$\text{dim}(V_{\text{off}})$}  &
\multicolumn{2}{c|}{  $\|u-u^{\text{off}} \|$ (\%) }  \\
\cline{2-3} {}&
$\hspace*{0.8cm}   L^{2}_\kappa(D)   \hspace*{0.8cm}$ &
$\hspace*{0.8cm}   H^{1}_\kappa(D)  \hspace*{0.8cm}$
\\
\hline\hline
       $834$     &       $35.58$    & $38.10$  \\
\hline
      $1512$    &    $14.34$    & $19.41$ \\
\hline
      $2084$    &    $6.81$    & $9.90$ \\
\hline
       $2306$   &   $4.54$  &$7.65$  \\

\hline
\end{tabular}
\end{table}%

In Tables \ref{table:HCC-stokes-Result} and \ref{table:Reverse of HCC_stokes Result}, we employ certain permeability
fields $\kappa^{-1}$ to get fast flow and Darcy flow simultaneously.  In Table \ref{table:HCC-stokes-Result}, we use
a permeability field $\kappa^{-1}$ small in inclusions, and large in the background, cf. Figure \ref{fig:perms}(c).
In this numerical test, we take $\lambda^{\text{off}}=1/3\mbox{, }1/4\mbox{, }1/7\mbox{, and }1/10$ with the offline
space of dimension 834, 1512, 2084 and 2316 respectively. From Table \ref{table:Reverse of HCCResult}, the energy
errors decrease from $38.10\%$ to $7.65\%$.
In Table \ref{table:Reverse of HCC_stokes Result}, we experimented with values of $\kappa^{-1}$ large in inclusions,
and small in the background as shown in Figure \ref{fig:perms}(d).  In this numerical test, we take
$\lambda^{\text{off}}=1/3\mbox{, }1/4\mbox{, }1/7\mbox{, and }1/10$ with the offline space of dimension 682,
1090, 1992 and 3344 respectively. From Table \ref{table:Reverse of HCCResult}, the energy errors decrease from $43.26\%$ to $5.38\%$.

\begin{table}[htb!]
\centering
\caption{Numerical results for problem \eqref{eq:Brinkman.fs} with $\kappa^{-1}$ in Figure \ref{fig:perms}(d). The $L^2$-weighted error and energy error are $74.68\%$ and $130.42\%$ for the MsFEM solution. In the simulation, the dimension of the snapshot space is fixed at 4498 with a weighted $L^2$ and energy relative error $1.47\%$ and $3.75\%$. }
 \label{table:Reverse of HCC_stokes Result}
\begin{tabular}{|c|c|c|c|c|}
\hline
\multirow{2}{*}{$\text{dim}(V_{\text{off}})$}  &
\multicolumn{2}{c|}{  $\|u-u^{\text{off}} \|$ (\%) }  \\
\cline{2-3} {}&
$\hspace*{0.8cm}   L^{2}_\kappa(D)   \hspace*{0.8cm}$ &
$\hspace*{0.8cm}   H^{1}_\kappa(D)  \hspace*{0.8cm}$
\\
\hline\hline
       $682$     &       $46.80$    & $43.26$  \\
\hline
      $1090$    &    $30.92$    & $30.30$ \\
\hline
      $1992$    &    $13.49$    & $13.15$ \\
\hline
       $3344$   &   $6.36$  &$5.38$  \\

\hline
\end{tabular}
\end{table}%

In Table \ref{table:HCC-stokes-Result}, the solution represents fast flow in the inclusions (with high
permeability value) and Darcy flow in the background, whereas in Table \ref{table:Reverse of HCC_stokes Result},
the solution is a fast flow in the background (with high permeability value) and Darcy flow in the inclusions.
The results in these four tables indicate that the errors are smaller when fast flow exists.

\section{Convergence analysis}
\label{Analysis}
%
In this section, we present a priori error estimates for the method. we first derive the stability argument.
Then we show the approximation property of the method. For the sake of simplicity, we assume a homogeneous
boundary condition ${g} = 0$ in the Brinkman equation \eqref{eq:Brinkman.fs}.
\subsection{Stability argument}\label{subsec:Stability argument}
To prove the stability of the method, we apply the well known {\em{inf-sup}} argument.
First, we define a norm on ${V}(\Om)$ by
\begin{equation}\label{eqn:Tnorm}
 \Tnorm{{u}}{\Om}^2=a({u},{u})+M\avrg{\DIV {u},\DIV {u}}{\Om},
\end{equation}
and the  norm on $W(\Om)$ is defined by
\begin{equation}\label{eqn:Snorm}
 \Snorm{p}{\Om}=M^{-\frac{1}{2}}\normL{p}{\Om},
\end{equation}
where $M=\max(\normI{\kappa^{-1}}{\Om},\,1)$. We also define the following two null spaces:
\begin{align*}
{Z} &:= \{{v} \in {V}(\Omega) : \; b({v}, p)= 0, \quad \mbox{ for all } \; p \in W(\Omega)\}, \\
{Z}^{\text{off}} &:= \{{v} \in {V}^{\text{off}} : \; b({v}, p)= 0, \quad
\mbox{ for all } \; p \in W^{\text{off}}\}.
\end{align*}

Under these definitions and the construction of $V^{\text{off}}$ and $W^{\text{off}}$, it holds:
\begin{equation}\label{eqn:coarsive_a}
{Z}^{\text{off}}  \subset {Z}, \quad a({v}, {v})  \succeq \Tnorm{{v}}{\Om} \quad \mbox{ for all } {v} \in {Z}.
\end{equation}
Here, and in what follows, we use the notation $A\succeq B$ to represent $A\geq \mathcal{C} B$ with
a constant $\mathcal{C}$ independent of the contrast and the functions involve, and a similar
interpretation applies to the notation $\preceq$.
The above two results imply that the bilinear form $a(\cdot, \cdot)$ is also coercive on ${Z}^{\text{off}}$.

We first verify that the continuous problem \eqref{eq:Brinkman.fs} satisfies the {\em inf-sup} condition.
\begin{lemma}\label{lemma:Cea'sLemma}
Let $\Tnorm{\cdot}{\Om}$ and $\Snorm{\cdot}{\Om}$ be defined in \eqref{eqn:Tnorm} and \eqref{eqn:Snorm}. Then
the following inf-sup condition holds independent of the contrast
 \begin{align}
\label{eqn:Cea'sLemma}
 \sup\limits_{{v}\in {V}(\Om)\backslash\{0\}}\,\frac{\avrg{\DIVinThm {v},q}{\Om}}{\Tnorm{{v}}{\Om}}\succeq \Snorm{q}{\Om},\,\mbox{ for all } q\in W(\Om).
\end{align}
\begin{proof}
It is well known \cite{BF} that the operator $b(\cdot, \cdot)$ satisfies the {\em inf-sup} condition under the standard norms, i.e.,
\begin{align}
\label{eqn:cea'sH1}
 \sup\limits_{{v} \in {V}(\Om)\backslash \{0\}}\,\frac{\avrg{\DIV {v}, q}{\Om}}{\normH{{v}}{\Om}}\succeq \normL{q}{\Om},\,\mbox{ for all } q\in W(\Om).
\end{align}
By the definition  of $\Tnorm{\cdot}{\Om}$ and $\Snorm{\cdot}{\Om}$, we have $\Tnorm{{v}}{\Om}\leq M^{\frac{1}{2}}\normH{{v}}{\Om}$, and
$\normL{q}{\Om}=M^{\frac{1}{2}}\Snorm{q}{\Om}$ for all $({v}, q) \in {V}(\Omega) \times W(\Omega)$. Combining these facts
with \eqref{eqn:cea'sH1} completes the proof.
\end{proof}

\end{lemma}
%
Next, we show that the discrete problem \eqref{eq:coarse_system} also satisfies this type of inf-sup condition
with a constant independent of the contrast for every offline space ${V}^{\text{off}} \times W^{\text{off}}$.
First, we consider the following auxiliary space:
\begin{align*}
{V}_H(\Omega) &:= \{{v} \in {H}^1(\Omega)| \; {v}|_{K} \in {Q}^2(K), \; \forall K \in \mathcal{T}_H\}, \\
W_H(\Omega) &:= \{q \in L^2_0(\Omega)| \; q|_{K} \in {Q}^0(K), \; \forall K \in \mathcal{T}_H\}.
\end{align*}
For the Brinkman equation, we have the following inf-sup condition in ${V}_H(\Om)\times W_H(\Om)$ (see \cite{Brenner&Scott}),
\begin{align}
\label{eqn:Cea'sLemma_Discrete}
 \sup\limits_{{v} \in {V}_H(\Om)\backslash\{0\}}\,\frac{\avrg{\DIVinThm {v} ,q}{\Om}}{\normH{{v}}{\Om}}\succeq \normL{q}{\Om},\,\mbox{ for all } q\in W_H(\Om).
\end{align}
Following the proof of Lemma \ref{lemma:Cea'sLemma}, we can obtain the discrete inf-sup condition in ${V}_H(\Om)\times W_H(\Om)$ with $\Tnorm{\cdot}{\Om}$ and $\Snorm{\cdot}{\Om}$, i.e.,
\begin{align}
\label{eqn:Cea'sLemma_Discrete_Triple}
 \sup\limits_{{v}\in {V}_H(\Om)\backslash\{0\}}\,\frac{\avrg{\DIVinThm {v} ,q}{\Om}}{\Tnorm{{v}}{\Om}}\succeq \Snorm{q}{\Om},\,\mbox{ for all } q\in W_H(\Om).
\end{align}

To prove the {\em inf-sup} condition for the space ${V}^{\text{off}}\times W^{\text{off}}$, we need the following result,
which states the stability of the Brinkman extension with respect to the weighted norm defined in \eqref{eqn:Tnorm}.
\begin{lemma}\label{thm:PropertyOfProjection}
For any ${w} \in ({H}^1(D))^d$, the Brinkman extension $\mathcal{H}({w})$ of $w$ on $D$ satisfies
\begin{align}
 \Tnorm{\mathcal{H}({w})}{D}\preceq\Tnorm{{w}}{D}.
\end{align}
\end{lemma}
\begin{proof}
By the definition of the Brinkman extension, $(\mathcal{H}({w}), p) \in ({H}(D))^d \times L^2_0(D)$ satisfies
\begin{align*}
&\nabla p-\Delta \mathcal{H}({w})+\kappa^{-1}\mathcal{H}({w})=0,\,\mbox{ in } D,&\\
&\DIV\,\mathcal{H}({w})=\frac{\int_{\partial D} {w} \cdot {n}}{|D|} \, \mbox{ in } D,&\\
 & \mathcal{H}({w})={w},\,\mbox{ on } \partial D.&
\end{align*}
Denote ${v}= \mathcal{H}({w})- {w}$, then ${v}$ satisfies
\begin{align}
\label{eqn:errorProjection}
&\nabla p-\Delta {v}+\kappa^{-1}{v}=\Delta {w}-\kappa^{-1}{w},\,\mbox{ in } D,&\\
\notag
&\DIV {v}= \frac{\int_{\partial D} {w} \cdot {n}}{|D|} - \DIV {w}\, \mbox{ in } D,&\\
\notag
 &{v}=0,\,\mbox{ on } \partial D.&
\end{align}
Since $p\in L^{2}_{0}(D)$, by Lemma 11.2.3 in \cite{Brenner&Scott}, there exists ${\phi} \in ({H}^{1}_{0}(D))^d$ such that
\begin{align}\label{eqn:pDivv}
p=-\DIV {\phi}  \quad \mbox{ and } \quad
\normH{{\phi}}{D}\preceq\normL{p}{D}.
\end{align}
Multiplying Equation\eqref{eqn:errorProjection} by $\phi$ and integrating by parts, we obtain,
%
\[
\avrg{ p,\,p}{D}+\avrg{\nabla {v},\,\nabla{\phi}}{D}+\avrg{\kappa^{-1}{v},\,{\phi}}{D}
=-\avrg{\nabla {w},\,\nabla {\phi}}{D}-\avrg{\kappa^{-1}{w},\,{\phi}}{D}.
\]
Thus
\[
\avrg{ p,\,p}{D}= -\avrg{\nabla {v},\,\nabla{\phi}}{D}- \avrg{\kappa^{-1}{v},\,{\phi}}{D}
-\avrg{\nabla {w},\,\nabla {\phi}}{D}-\avrg{\kappa^{-1}{w},\,{\phi}}{D}.
\]
Using the Cauchy-Schwarz inequality and \eqref{eqn:pDivv}, we arrive at,
\begin{equation*}
  \begin{aligned}
 \normL{p}{D}^{2} 
&\preceq
(\normL{\nabla {v} }{D}+\norm{\kappa^{-1}{v}}_{H^{-1}(D)}+
(\normL{\nabla {w} }{D}+\norm{\kappa^{-1}{w}}_{H^{-1}(D)})\normL{p}{D}\\
&\preceq(\normL{\nabla {v} }{D}+M\norm{\kappa^{-\frac{1}{2}}{v}}_{L^{2}(D)}
+\normL{\nabla {w} }{D}+M\norm{\kappa^{-\frac{1}{2}}{w}}_{L^{2}(D)})\normL{p}{D}.
\end{aligned}
\end{equation*}
Then it yields the pressure estimate
\begin{align}\label{eqn:pressureProjection}
\normL{p}{D}\preceq&\normL{\nabla {v} }{D}+M\norm{\kappa^{-\frac{1}{2}}{v}}_{L^{2}(D)}
+\normL{\nabla {w} }{D}+M\norm{\kappa^{-\frac{1}{2}}{w}}_{L^{2}(D)}.
\end{align}
Multiplying Equation\eqref{eqn:errorProjection} by ${v}$ and integrating by parts, yields,
\begin{align*}
-\avrg{p,\DIV {v}}{D}+\avrg{\nabla {v},\nabla {v}}{D}+\avrg{\kappa^{-1}{v},{v}}{D}=\avrg{\Delta {w}-\kappa^{-1}{w},{v}}{D}.
\end{align*}
Using Cauchy-Schwarz inequality and the fact that $p$ has zero mean
on $D$, it follows that,
\begin{align*}
 &\avrg{\nabla {v},\nabla {v}}{D}+\avrg{\kappa^{-1}{v},{v}}{D}=\avrg{\Delta {w}-\kappa^{-1}{w},{v}}{D}+\avrg{p,\DIV\,{v}}{D}&\\
&=\avrg{\Delta {w}-\kappa^{-1}{w},{v}}{D}+\avrg{p,\DIV\,(\mathcal{H}({w})-{w})}{D}&\\
&=\avrg{\Delta {w}-\kappa^{-1}{w},{v}}{D} - \avrg{p,\DIV \, {w}}{D}&\\
&\leq \normL{\nabla {w} }{D}\normL{\nabla {v} }{D}+\normL{\kappa^{-\frac{1}{2}} {w}}{D}\normL{\kappa^{-\frac{1}{2}}{v}}{D}
+\normL{p}{D}\normL{\DIV {w}}{D}.&
\end{align*}
Inserting the pressure estimate \eqref{eqn:pressureProjection} and Young's inequality, we deduce
\begin{equation*}
\begin{aligned}
 \avrg{\nabla {v},\nabla {v}}{D}+\avrg{\kappa^{-1}{v},{v}}{D}
 &\leq \frac{1}{2\delta}(\normL{\nabla {w} }{D}^{2}+\normL{\kappa^{-\frac{1}{2}}{w}}{D}^{2})
+\frac{\delta}{2}\left(\normL{\nabla {v}}{D}^{2}+\normL{\kappa^{-\frac{1}{2}}{v}}{D}^{2}\right)\\
&+\frac{\delta}{2M}\left(\normL{\nabla {v} }{D}^{2}+M\norm{\kappa^{-\frac{1}{2}}{v}}_{L^{2}(D)}^{2}
+\normL{\nabla {w}}{D}^{2}\right.\\
& \left.+M\norm{\kappa^{-\frac{1}{2}}{w}}_{L^{2}(D)}^{2}\right)
+\frac{M}{2\delta}\normL{\DIV {w}}{D}^{2}.
\end{aligned}
\end{equation*}
Now the choice $\delta=\frac{1}{4}$ yields
\begin{align*}
\avrg{\nabla {v},\nabla {v}}{D}+\avrg{\kappa^{-1}{v},{v}}{D}
\preceq \normL{\nabla {w} }{D}^{2}+\normL{\kappa^{-\frac{1}{2}}{w}}{D}^{2}
+M\normL{\DIV {w}}{D}^{2}=\Tnorm{{w}}{D}^{2}.
\end{align*}
Recall that $ \mathcal{H}({w})={v}+{w}$. By triangle inequality, we have
\begin{align*}
 \avrg{\nabla \mathcal{H}({w}),\nabla \mathcal{H}({w})}{D}+\avrg{\kappa^{-1}\mathcal{H}({w}),\mathcal{H}({w})}{D}
\preceq\Tnorm{{w}}{D}^{2}.
\end{align*}
It suffices to show
\begin{align}
\label{eqn:LeftToBeShow}
 M^{\frac{1}{2}}\normL{\DIV \mathcal{H}({w})}{D}\preceq\Tnorm{{w}}{D}.
\end{align}
Indeed from the compatibility condition, we obtain: $\DIV \mathcal{H}({w})=\frac{1}{|D|}\int\limits_{D}\,\DIV{w}$. Hence,
\begin{align*}
 |\DIV \mathcal{H}({w})|=|\frac{1}{|D|}\int\limits_{D}\,\DIV {w}|\leq \frac{1}{|D|}\int\limits_{D}\,|\DIV {w}|
\leq\normL{\DIV {w}}{D}|D|^{-\frac{1}{2}},
\end{align*}
where in the last step we used Cauchy-Schwarz inequality. Consequently
\begin{align*}
 \normL{\DIV \mathcal{H}({w})}{D}^{2}\preceq \normL{\DIV {w}}{D}^{2}|D|^{-1}|D|=\normL{\DIV {w}}{D}^{2},
\end{align*}
This completes the proof.
\end{proof}

We are now ready to show the {\em inf-sup} condition in the space ${V}^{\text{off}} \times W^{\text{off}}$.
\begin{lemma}\label{lemma:Cea'sLemma_off}
For $\Tnorm{\cdot}{\Om}$ and $\Snorm{\cdot}{\Om}$ defined in \eqref{eqn:Tnorm} and \eqref{eqn:Snorm}, we
have the following inf-sup condition with inf-sup constant independent of the contrast
 \begin{align}
\label{eqn:Cea'sLemma_DiscreteForm}
 \sup\limits_{{v} \in {V}^{\text{off}}(\Om)\backslash\{0\}}\frac{\avrg{\DIVinThm {v},q}{\Om}}{\Tnorm{{v}}{\Om}}\succeq \Snorm{q}{\Om},\,\mbox{ for all } q\in W^{\mathrm{off}}(\Om).
\end{align}
\end{lemma}
\begin{proof}
First note $W^{\text{off}} = W_H(\Omega)$. By \eqref{eqn:Cea'sLemma_Discrete_Triple}, we have
\[
 \sup\limits_{{v}\in {V}_H(\Om)\backslash\{0\}}\,\frac{\langle\mathrm{div}\ {v} ,q\rangle_{\Om}}{\Tnorm{{v}}{\Om}}\succeq \Snorm{q}{\Om},\,\mbox{ for all } q\in W^{\text{off}}.
\]
For any ${v} \in {V}_H(\Omega)$, let $\mathcal{H}({v})$ be the Brinkman extension of
${v}|_{\mathcal{E}_H}$, i.e., $\mathcal{H}({v})$ takes the value of ${v}$ on the skeleton $\mathcal{E}_H$ and
is extend to the interior by Brinkman extension within each coarse block. Then
${v}|_{F} \in [P^2(F)]^2, \; \forall F \in \mathcal{E}_H$. According to the construction
of the offline space ${V}^{\text{off}}$ in Section 3, we have
\[
   \mathcal{H}({v}) \in {V}^{\text{off}}.
\]
Moreover, for any $q \in W^{\text{off}}$, $q$ is piecewise constant on each coarse block.
By combining this fact and the definition of Brinkman extension, we have
\[
\avrg{\DIVinThm {v} ,q}{D} = \avrg{\DIVinThm \mathcal{H}({v}) ,q}{D},
\]
for every coarse block $D$. Finally, we complete the proof by using Lemma \ref{thm:PropertyOfProjection}:
\begin{align*}
\Snorm{q}{\Om} &\preceq  \sup\limits_{{v} \in {V}_H(\Om)\backslash\{0\}}\frac{\avrg{\DIVinThm {v},q}{\Om}}{\Tnorm{{v}}{\Om}} = \sup\limits_{{v} \in {V}_H(\Om)\backslash\{0\}}\frac{\avrg{\DIVinThm \mathcal{H}({v}),q}{\Om}}{\Tnorm{{v}}{\Om}}\\
&\preceq  \sup\limits_{{v} \in {V}_H(\Om)\backslash\{0\}}\frac{\avrg{\DIVinThm \mathcal{H}({v}),q}{\Om}}{\Tnorm{\mathcal{H}({v})}{\Om}}
\preceq  \sup\limits_{{v} \in {V}^{\text{off}}(\Om)\backslash\{0\}}\frac{\avrg{\DIVinThm {v},q}{\Om}}{\Tnorm{{v}}{\Om}}.
\end{align*}

\end{proof}

Now by combining Lemma \ref{lemma:Cea'sLemma}, Lemma \ref{lemma:Cea'sLemma_off} and \eqref{eqn:coarsive_a},
we obtain the following stability result, by repeating the proof of Theorem 3.2 in \cite{XXX08}.

\begin{theorem}\label{thm:errorEstimate-offline}
Let $({u}, p) \in {V}(\Omega)\times W(\Omega)$ and $({u}_0, p_0) \in {V}^{\mathrm{off}}(\Om)\times
W^{\mathrm{off}}(\Om)$ be the Galerkin solutions of problem \eqref{eq:Brinkman.fs} and problem \eqref{eq:coarse_system} respectively. We have
 \begin{align}
\label{eqn:ErrForVelocity-offline}
 &\Tnorm{{u}-{u}_0}{\Om}\preceq \inf\limits_{{w} \in {V}^{\text{off}}(\Om)}\Tnorm{{u}-{w}}{\Om},&
\end{align}
\end{theorem}

\subsection{Convergence results}
Now we  derive an error estimate for our method.
To this end, we first give several basic estimates on the Brinkman extension.
\begin{lemma}\label{thm:error1}
For each partition of unity function $\chi_i$ with support $\omega_i$, let $({u}_c, p_c)
\in ({H}^1(\omega_i))^d \times L^2_0(\omega_i)$ solve
\begin{alignat*}{5}
\nabla p_c-\Delta {u}_c +&\kappa^{-1} {u}_c& =&0& \mbox{ in } \omega_i,\\
 &\DIVinThm {u}_c& =& \frac{\int_{\partial\omega_i}\,{g} \cdot n}{|\omega_i|}& \mbox{ in } \omega_i ,\\
&\quad\;{u}_c&=& {g} &\;\mbox{ on } \partial\omega_i.
\end{alignat*}
Then the following a priori estimate holds
\begin{equation}\label{eqn:cacciopoli_pre}
 \int_{\omega_i}\,\chi_{i}^2 |\nabla {u}_c|^2+\int_{\omega_i}\,\kappa^{-1}\chi_{i}^2|{u}_c|^2
\preceq\int_{\omega_i}\,|\nabla\chi_{i}|^2 |{u}_c|^2
+\int_{\omega_i}\,\kappa^{-2}|{u}_c|^2+\int_{\omega_i}\,{|\DIVinThm {u}_c}|^2+\normL{p_c}{\omega_i}^{2}.
\end{equation}
\end{lemma}
\begin{proof}
Multiplying the equation by $\chi_{i}^2 {u}_c$ yields
\[
 -\avrg{p_c,\DIV(\chi_i^2 {u}_c)}{\omega_i}+\avrg{\nabla {u}_c,\nabla (\chi_{i}^2{u}_c)}{\omega_i}
+\avrg{\kappa^{-1} {u}_c, \chi_{i}^2 {u}_c}{\omega_i}=0.
\]
Some simple algebraic manipulations give
\begin{align*}
\int\limits_{\omega_i}\, \chi_{i}^2 &|\nabla {u}_c|^2
+\int\limits_{\omega_i}\,\kappa^{-1}\chi_{i}^2 {u}_c^2 =  \avrg{p_c,2\chi_{i}\nabla\chi_{i} \cdot {u}_c}{\omega_i}
+\avrg{p_c, \chi_{i}^{2}\DIV {u}_c}{\omega_i}- \avrg{\nabla {u}_c ,2\chi_{i} \nabla \chi_{i} \cdot {u}_c}{\omega_i},&\\
&\preceq \normL{p_c}{\omega_i}(\normL{\nabla \chi_{i}\cdot {u}_c}{\omega_i}+\normL{\DIV  {u}_c}{\omega_i})
+\normL{\chi_{i}\nabla {u}_c}{\omega_i}  \normL{\nabla \chi_{i}\cdot {u}_c}{\omega_i},&\\
&\leq\frac{\delta}{2} (\normL{p_c}{\omega_i}^{2}+\normL{\chi_{i}\nabla {u}_c}{\omega_i}^{2})
+\frac{1}{2\delta}(\normL{\nabla\chi_{i}\cdot {u}_c}{\omega_i}^{2} +\normL{\DIV {u}_c}{\omega_i}^2).&
\end{align*}
Taking $\delta=\frac{1}{4}$, we obtain the desired inequality.
\end{proof}
\begin{lemma}\label{thm:error2}
Let $\omega_i \subset \mathcal{T}_H$ be an arbitrary coarse neighborhood. Let $({u}_N, p_N) \in ({H}^1_{0}(\omega_i))^d \times L^2_0(\omega_i)$ solve
\begin{alignat*}{5}
\nabla p_N -\Delta {u}_N+&\kappa^{-1} {u}_N& =&{f}& \mbox{ in } \omega_i,\\
&\DIVinThm {u}_N&=&0& \mbox{ in } \omega_i,  \\
&\quad\;{u_{N}}&=&0& \mbox{ on } \partial \omega_i.
\end{alignat*}
Then there holds
\begin{align}\label{eqn:dirichletResult}
\Tnorm{{u}_N}{\omega_i}\preceq H\normL{{f}}{\omega_i}.
\end{align}
\end{lemma}
\begin{proof}
By multiplying the first equation by ${u}_N$, integrating by parts and  the divergence free property of ${u}_N$, we obtain
\begin{align*}
\normL{\nabla {u}_N}{\omega_i}^{2}+\norm{\kappa^{-\frac{1}{2}}{u}_N}_{L^{2}(\omega_i)}^{2}=\avrg{{f},\,{u}_N}{\omega_i}.
\end{align*}
In view of the boundary condition, we can apply Poincar\'{e}'s inequality,
\[
\normL{{u}_N}{\omega_i}\preceq H\normL{\nabla {u}_N}{\omega_i}.
\]
Thus
\begin{align*}
\normL{\nabla {u}_N}{\omega_i}^{2}+\norm{\kappa^{-\frac{1}{2}}{u}_N}_{L^{2}(\omega_i)}^{2}=\avrg{{f},\,{u}_N}{\omega_i}\preceq H\normL{\nabla {u}_N}{\omega_i}\normL{{f}}{\omega_i}.
\end{align*}
Finally, we complete the proof by the young's inequality.
\end{proof}
Now we are ready to state our main error estimate.
\begin{theorem}\label{thm:finalResult}
Let $\Lambda_*=\min\limits_{\omega_i}\lambda^{\omega_i}_{L_i+1}$. Then
\begin{align*}
  \Tnorm{u-u_0}{\Omega}^{2}\preceq\frac{1}{\Lambda_*}\normL{\nabla  u}{\Omega}^2+H^2\normL{f(x)}{\Omega}^2+\normL{p_c}{\Omega}^{2}.
\end{align*}
where $p_c$ is defined by \eqref{eqn:res_pre} below.

\end{theorem}
\begin{proof}
In view of the linearity of the equation \eqref{eq:Brinkman.fs}, on each coarse neighborhood $\omega_i \subset \mathcal{T}_H$, $u$ can be decomposed
into $u=\mathcal{H}({u})+{u}_N$, where $\mathcal{H}({u})$ is the Brinkman extension of ${u}$ and ${u}_N$ is the
residual in Lemma \ref{thm:error2}. For each $\chi_i$, let $I^{0}{u}$ be the
local interpolant of ${u}$ in the local offline space $\tilde{{V}}^{\omega_i}_{\text{off}}$. Then there exists $p_c\in L^{2}(\omega_i)$, s.t.
\begin{alignat}{5}
\label{eqn:res_pre}
\nabla p_c-\Delta ({u}-I^{0}{u}) +&\kappa^{-1} ({u}-I^{0}{u})& =&0& \mbox{ in } \omega_i,\\
\notag
 &\DIVinThm ({u}-I^{0}{u})& =& \frac{\int_{\partial\omega_i}\,{h_i} \cdot n}{|\omega_i|}& \mbox{ in } \omega_i ,\\
\notag
&\quad\;({u}-I^{0}{u})&=& {h_i} &\;\mbox{ on } \partial\omega_i,
\end{alignat}
since $I^{0}{u}$ equals 0 over $\partial\omega_i$ (the support of $\chi_i$ is $\omega_i$) and each basis in $\tilde{{V}}^{\omega_i}_{\text{off}}$ has the properties of divergence constant. Here, $h_i$ denotes the boundary value of ${u}-I^{0}{u}$ over $\partial\omega_i$.

By the construction of the offline space $V^\text{off}$, $\mathcal{H}(\chi_i I^0 {u}) \in {V}^{\text{off}}$. By
Theorem \ref{thm:errorEstimate-offline}, we have
\begin{align*}
 \Tnorm{{u}-{u_0}}{\Om}^{2}&\preceq\inf\limits_{{v} \in {V}^{\text{off}}}\,\Tnorm{{u}-{v}}{\Om}^{2}\\
  &\preceq\TnormLer{{u}-\sum^{N_C}_{i=1}\mathcal{H}(\chi_{i}I^{0}{u})}{\Om}^{2}\preceq\TnormLer{\mathcal{H}({u})-
    \sum^{N_c}_{i=1}\mathcal{H}(\chi_{i}I^{0}{u})}{\Om}^{2}+\Tnorm{{u}_N}{\Om}^{2},\\
  &\preceq\TnormLer{\mathcal{H}(\sum^{N_c}_{i=1}\chi_{i}{u})-\sum^{N_c}_{i=1}\mathcal{H}(\chi_{i}I^{0}{u})}{\Om}^{2}+H^2\normL{f(x)}{\omega_i}^2.
\end{align*}
Here the last step follows from the estimate in Lemma \ref{thm:error2}.
For the first term, we have
\begin{align*}
&\TnormLer{\mathcal{H}(\sum^{N_c}_{i=1}\chi_{i}{u})-\sum^{N_c}_{i=1}\mathcal{H}(\chi_{i}I^{0}{u})}{\Om}^{2}
=\TnormLer{\mathcal{H}(\sum^{N_c}_{i=1}(\chi_{i}{u} - \chi_{i}I^{0}{u})}{\Om}^{2}&\\
&\preceq \sum_{i=1}^{N_c} \Tnorm{\mathcal{H}(\chi_i {u} - \chi_i I^0 {u})}{\omega_i}^{2} \preceq \sum_{i=1}^{N_c} \Tnorm{\chi_i ({u} - I^0 {u})}{\omega_i}^{2}, &
\end{align*}
where at the last step we have applied Lemma \ref{thm:PropertyOfProjection} on each coarse neighborhood $\omega_i$.
Consequently,
 \begin{align*}
 \Tnorm{{u} -{u}_0}{\Om}^{2}&\preceq
 \sum\limits_{i=1}^{N_c}\Tnorm{\chi_{i}({u}-I^{0}{u})}{\omega_i}^{2}+H^2\normL{{f}}{\omega_i}^2&\\
&\preceq\sum\limits^{N_c}_{i=1}\int\limits_{\omega_i}\,\chi_{i}^2|\nabla({u}-I^{0}{u})|^2
+\int\limits_{\omega_i}\,\kappa^{-1}\chi_{i}^2|{u}-I^{0}{u}|^2&\\
&+M\int\limits_{\omega_i}\,\chi_{i}^2|\DIV ({u}- I^{0}{u})|^2+M\int\limits_{\omega_i}|\nabla \chi_{i}|^2|{u}-I^{0}{u}|^2+H^2\normL{{f}}{\omega_i}^2.&
\end{align*}
By applying Lemma \ref{thm:error1} to the term ${u} - I^{0} {u}$ in Eqn. \eqref{eqn:res_pre}, we deduce
\begin{align*}
 \Tnorm{{u}- {u}_0}{\Omega}^{2}
 \preceq &\sum\limits_{i}M\int_{\omega_i}|\nabla\chi_{i}|^2|{u}-I^{0}{u}|^2+\int_{\omega_i}\,(\kappa^{-1})^2|{u} - I^{0} {u}|^2&\\
&+M\int_{\omega_i}\, \chi_{i}^2|\DIV{({u}-I^{0}{u})}|^2+H^2\normL{{f}}{\omega_i}^2+\normL{p_c}{\omega_i}^{2}.&
\end{align*}
%

%
%
Finally, using the spectral problem \eqref{eqn:eig}, with $A$ and $S$ defined by
\begin{equation}
\label{eqn:eig_analysis}
  \begin{aligned}
    \displaystyle A &= [a_{mn}] = \int_{\omega_i} (\chi_i)^2 \nabla {\psi}_m^{\omega,\text{snap}} \cdot \nabla {\psi}_n^{\omega,\text{snap}},\\
  \displaystyle S &= [s_{mn}] = \int_{\omega_i}(\kappa(x)^{-2} +M (\nabla\chi_i)^2) {\psi}_m^{\omega,\text{snap}}\cdot {\psi}_n^{\omega,\text{snap}}+
M\int_{\omega_i} (\chi_i)^2 \DIV {\psi}_m^{\omega,\text{snap}}  \DIV {\psi}_n^{\omega,\text{snap}},
   \end{aligned}
 \end{equation}
we have
\begin{align*}
&\int_{\omega_i}\,M(\nabla\chi_{i})^2|{u}-I^{0}{u}|^2+\int_{\omega_i}\,(\kappa^{-1})^2|{u}-I^{0}{u}|^2+M\int_{\omega_i}\,(\chi_{i})^2|\DIV{({u}-I^{0}{u})}|^2&\\
&\leq\frac{1}{\lambda^{\omega_i}_{L_i+1}}\int_{\omega_i}\,(\chi_{i})^2|\nabla({u}-I^{0}{u})|^2.&
\end{align*}
Hence,
\begin{align*}
 \Tnorm{u-u_0}{\Omega}^{2}\preceq\sum\limits_{i}\frac{1}{\lambda^{\omega_i}_{L_i+1}}\int_{\omega_i}\,
(\chi_{i})^2|\nabla(u-I^{0}u)|^2+H^2\normL{f(x)}{\omega_i}^2+\normL{p_c}{\omega_i}^{2}.
\end{align*}
Upon denoting $\Lambda_*=\min\limits_{\omega_i}\lambda^{\omega_i}_{L_i+1}$, we deduce
\begin{align*}
  \Tnorm{u-u_0}{\Omega}^{2}\preceq\frac{1}{\Lambda_*}\sum\limits_{i}\int_{\omega_i}\,
(\chi_{i})^2|\nabla(u-I^{0}u)|^2+H^2\normL{f(x)}{\omega_i}^2+\normL{p_c}{\omega_i}^{2}.
\end{align*}
Using the inequality $\normL{\nabla I^{0} u}{\omega_i}\preceq \normL{\nabla  u}{\omega_i}$,
\begin{align*}
  \Tnorm{u-u_0}{\Omega}^{2}\preceq\frac{1}{\Lambda_*}\sum\limits_{i}\normL{\nabla  u}{\omega_i}^2+H^2\normL{f(x)}{\omega_i}^2+\normL{p_c}{\omega_i}^{2},
\end{align*}
and thus
\begin{align*}
  \Tnorm{u-u_0}{\Omega}^{2}\preceq\frac{1}{\Lambda_*}\normL{\nabla  u}{\Omega}^2+H^2\normL{f(x)}{\Omega}^2+\normL{p_c}{\Omega}^{2}.
\end{align*}
This completes the proof of the theorem.
\end{proof}
\begin{remark}
We note that in the analysis, we have used the spectral problem \eqref{eqn:eig_analysis}, instead of \eqref{eqn:eig}
in the numerical simulation. In view of the inequality $\normL{\DIV u}{D}\leq \normL{\nabla u}{D}$ for any
$u\in (H^{1}(D))^d$ and the fact that $\chi_i$ is bounded, these two spectral problems are equivalent provided
that $M$ is bounded. Hence our analysis does provide partial justification for the algorithm. The constant $M$
appears as a result of the definition of the velocity and pressure norms, cf. \eqref{eqn:Tnorm} and \eqref{eqn:Snorm},
which is needed for the inf-sup condition.
It remains unclear how to get rid of the constant $M$ in the norm definition in the convergence analysis.
\end{remark}

\section{Conclusion}
\label{sec:conclusion}
In this work, we have developed a mixed generalized multiscale finite element method for
the Brinkman flow in high-contrast media, which is able to capture both the Stokes flow and
the Darcy flow in respective regions. In the fine grid, we approximate the
velocity and pressure with piecewise quadratic and piecewise constant functions.
We develop a novel approach to construct a coarse approximation for the velocity
snapshot space, and a robust low-dimensional offline space for the velocity.
The stability of the mixed GMsFEM and a priori error estimates are derived. The
two-dimensional numerical examples illustrate clearly the robustness and efficiency
of the method.

In our discussion, we have focused on the approximation of the velocity space, and
simply taken the piecewise constant space for the pressure. This may not be the
best choice, as can be seen from Thm \ref{thm:finalResult}. The mixed finite
element space may get better results with a better pressure space and accordingly
an enriched velocity space. We leave the enriching of pressure space to a future work.
Further, it is natural to extend the proposed method to
the Stokes model in perforated domains.

\section{Acknowledgements}
 G. Li's research is partially supported by DOE. 

\bibliographystyle{siam}
\bibliography{references}

\end{document}